\documentclass[11pt]{amsart}

\usepackage{pdfpages}
\usepackage{amssymb}
\usepackage{amsmath}
\usepackage{amsthm}
\usepackage{thm-restate}
\usepackage{graphicx}
\usepackage{float}
\usepackage{tikz}
\graphicspath{ {./pics/} }

\newtheorem{theorem}{Theorem}

\newtheorem{lemma}[theorem]{Lemma}
\newtheorem{definition}[theorem]{Definition}
\newtheorem{proposition}[theorem]{Proposition}
\newtheorem*{conjecture}{Conjecture}
\newtheorem*{theorem*}{Theorem}
\newtheorem*{remark}{Remark}

\numberwithin{theorem}{section} % numbers theorems using the subsection number
\numberwithin{corollary}{section}
\numberwithin{lemma}{section}
\numberwithin{proposition}{section}
\numberwithin{definition}{section}

\newtheorem{innercustomthm}{Theorem}
\newenvironment{customthm}[1]
	{\renewcommand\theinnercustomthm{#1}\innercustomthm}
	{\endinnercustomthm}

\begin{document}
\title{Maximizing Core Entropy}

\author{Jacob Kewarth}
\address{University of Toronto, 40 St. George Street, M5S 2E4, Toronto, Canada}
\email{jacob.kewarth@mail.utoronto.ca}

\begin{abstract}
Here we study the core entropy function $h: PM(d) \rightarrow \mathbb{R}$ defined on the space of degree $d$ primitive majors, a combinatorial model for the space of degree $d$ polynomials. In particular, we introduce techniques to classify the global maximum for each $d$ and the maxima along certain strata. We show that there are $d-1$ global maxima over all of $PM(d)$ and that the maxima on the unicritical stratum is a Cantor set.
\end{abstract}

\maketitle

%\tableofcontents

\section{Introduction}

For a postcritically finite polynomial $f$, the core entropy of $f$ is the topological entropy of $f$ restricted to its Hubbard tree. This dynamical invariant was studied by Thurston (see \cite{Th+}) who raised many questions and conjectures which influenced subsequence research. Core entropy exhibits interesting interactions both with the dynamics of $f$ itself and with the parameter space $f$ lives in. With respect to the Julia set $J_f$ of $f$, the core entropy of $f$ is a constant multiple of the Hausdorff dimension of the set of biaccessible angles. In the space of postcritically finite polynomials core entropy varies continuously (see \cite {Ti1} or \cite{DS} for quadratics and \cite{GT} for general degree) and in the Mandelbrot set it is monotonic along veins (see \cite{Li} and \cite{Ti1}). This last property has led to considerable interest in the maxima and local maxima of core entropy. For quadratic polynomials, the local maxima were conjectured to occur at dyadic angles (with the unique global maxima at $\theta = 1/2$). Partial progress on this conjecture was made in \cite{Ju} and a full proof was given in \cite{DS}. In \cite{GT} Gao and Tiozzo asked for a description of the global maxima for core entropy in higher degrees and the goal of this paper is to develop techniques to answer their question.

A common approach in studying core entropy is to associate to each postcritically finite polynomial a rational primitive major (see \cite{Po}). Recall that a primitive major of degree $d$ is a set of ideal polygons (called leaves) $M = \{ L_1, L_2, ..., L_r \}$ on the closed unit disc such that:
\begin{enumerate}
\item The elements of $M$ are pairwise disjoint,
\item If we parameterize the unit circle $\partial \mathbb{D}$ by $[0, 1)$ then for each leaf $L_j$, the endpoints $\partial \mathbb{D} \cap L_j$ are identified under $T(z) = dz \mod 1$,
\item $\sum_{j=1}^r (|\partial \mathbb{D} \cap L_j | -1) = d-1$.
\end{enumerate}

Then $PM(d)$ is the set of all primitive majors of degree $d$. We say that a $M$ is a rational primitive major if the endpoints of all of its leaves are rational numbers and denote the space of all such majors as $RPM(d)$. Note that the endpoints of a leaf being rational implies they are (eventually) periodic under $T(z) = dz \mod 1$.

The leaves of the rational primitive major act as a combinatorial model for the critical points of the polynomial, and the iterations of the endpoints of the leaves by $T(z) = dz \mod 1$ model the iterations of the critical points. Thus we can instead define core entropy as a function on $RPM(d)$. In fact we can continuously extend core entropy to a map defined on all $PM(d)$.

Then our first main result of the paper is:

\begin{customthm}{A}
For any $d \geqslant 2$ the set of maxima of core entropy over $PM(d)$ contains $d-1$ elements. When $d$ is even all these maxima are identified under $S(z) = z + \frac{1}{d-1} \mod 1$. When $d$ is odd there are two equivalence classes under $S$.
\end{customthm}

\begin{figure}[!h]
\centering
\includegraphics[scale=0.4]{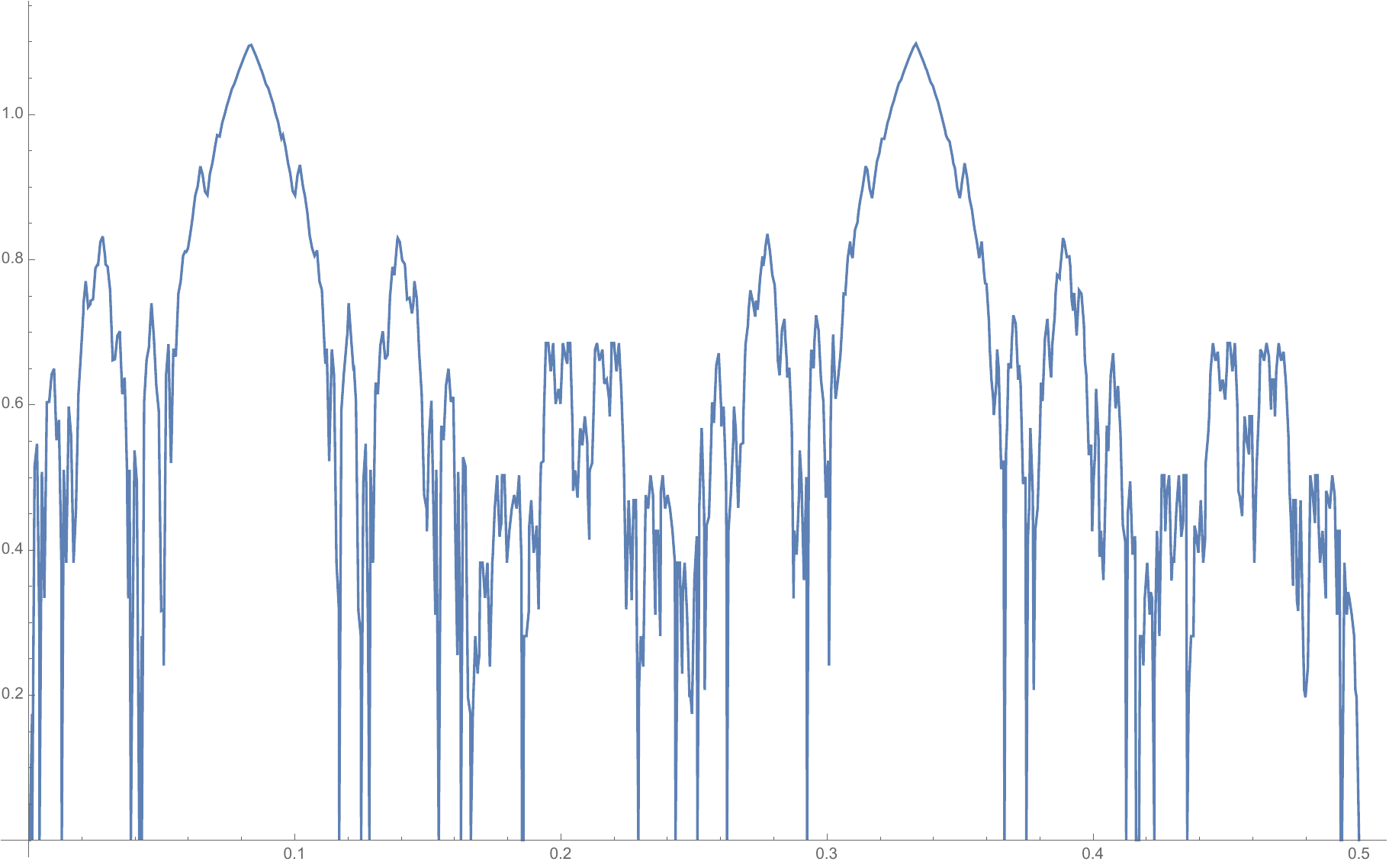}
\caption{A slice of the graph of core entropy for $PM(3)$ highlighting the 2 global maxima. This figure is from [GT].}
\end{figure}

For a fixed $d$, elements of $PM(d)$ can vary wildly. For example, in $PM(3)$ an element could consist of either two ideal lines or a single ideal triangle. Thus it is natural to subdivide $PM(d)$ into strata.  If we have natural numbers $s_j$ such that $\sum (s_j - 1) = d-1$ then we can define the stratum $\Pi(s_1, ..., s_r)$ to be the subset of $PM(d)$ with leaves of size $s_1, ..., s_r$. In particular, for each $d$ we have the unicritical stratum $\Pi(d)$ consisting of a single ideal d-gon. 

Our second main result states:

\begin{customthm}{B}
For any $d \geqslant 3$ the set of maxima of core entropy over $\Pi(d)$ is a Cantor set.
\end{customthm}

\begin{figure}[!h]
\centering
\includegraphics[scale=0.4]{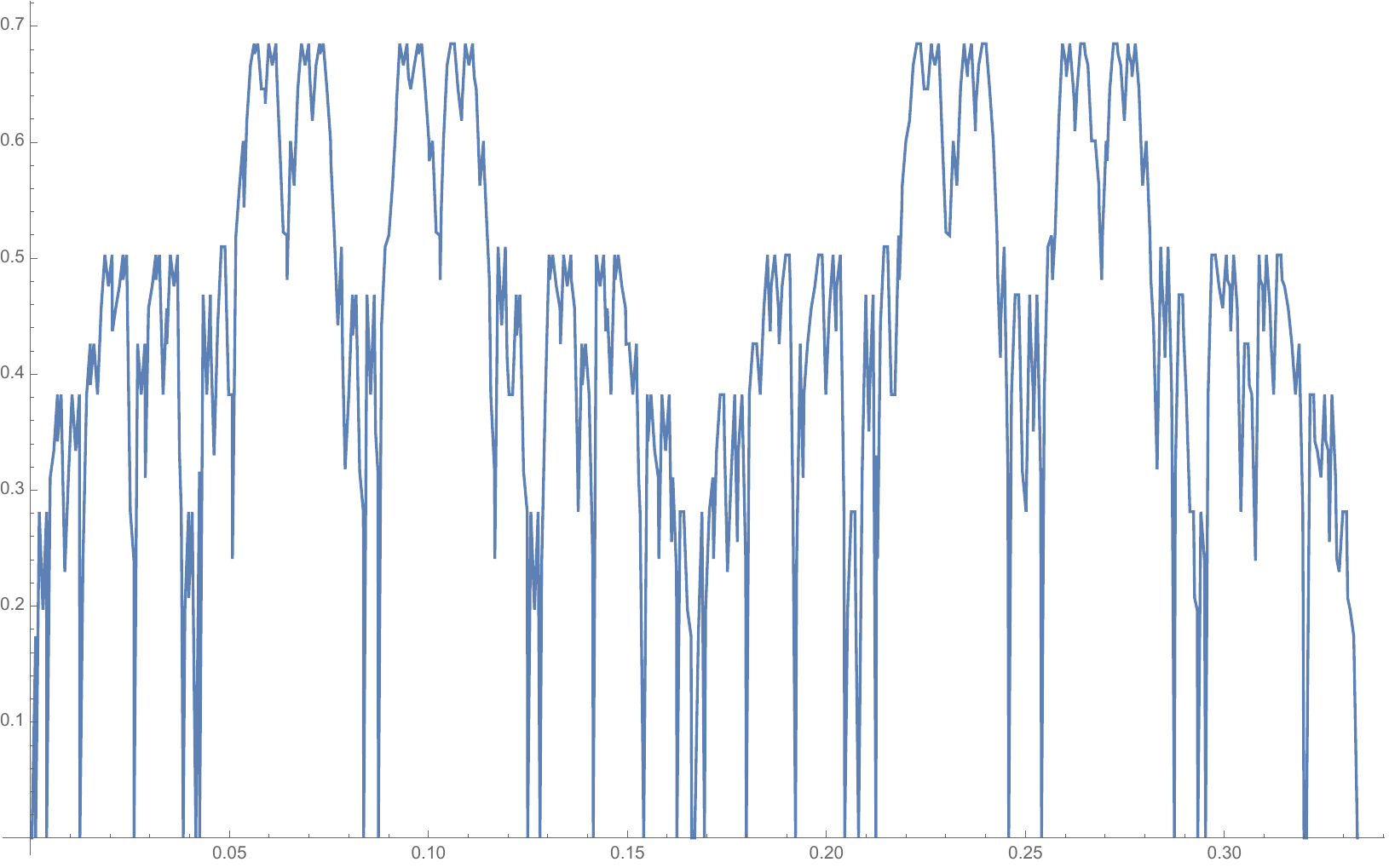}
\caption{The graph of core entropy on the cubic unicritical stratum $\Pi(3)$. This figure is from [GT].}
\end{figure}

In section $2$ we review relevant preliminaries. In particular we discuss two algorithms for computing/defining the core entropy function on $PM(d)$.

In section $3$ we introduce the notion of a primitive major being totally separated, a simple combinatorial condition which we show is equivalent to maximizing entropy.

In sections $4$ and $5$ we describe the maxima of core entropy over all of $PM(d)$ and over the unicritical stratum.

Finally in section $6$ we briefly discuss how to generalize these ideas to arbitrary stratum of $PM(d)$.

\section{Preliminaries}

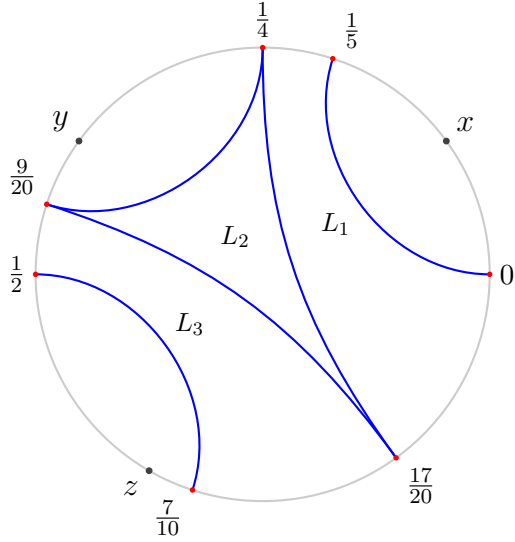
\begin{figure}
\begin{tikzpicture}[scale=3]

  % 1. Boundary Circle
  \draw[thick, gray!40] (0,0) circle (1);

  % 2. Main Lamination Leaf Vertices (Parameterized [0,1) -> mapped to 360 degrees)
  \path (0*360:1)         coordinate[label={0*360:$0$}]           (L1_A)
        (1/5*360:1)       coordinate[label={1/5*360:$\frac{1}{5}$}] (L1_B)
        
        (1/4*360:1)       coordinate[label={1/4*360:$\frac{1}{4}$}] (T_A)
        (9/20*360:1)      coordinate[label={9/20*360:$\frac{9}{20}$}] (T_B)
        (17/20*360:1)     coordinate[label={17/20*360:$\frac{17}{20}$}] (T_C)
        
        (1/2*360:1)       coordinate[label={1/2*360:$\frac{1}{2}$}] (L2_A)
        (7/10*360:1)      coordinate[label={7/10*360:$\frac{7}{10}$}] (L2_B);

  % 3. Isolated Marked Target Points
  \path (1/10*360:1)      coordinate[label={1/10*360:\textbf{\large $x$}}] (Pt_X)
        (8/20*360:1)      coordinate[label={8/20*360:\textbf{\large $y$}}] (Pt_Y)
        (2/3*360:1)       coordinate[label={2/3*360:\textbf{\large $z$}}] (Pt_Z);

  % 4. Draw Leaves (Ideal Geodesics) with Exact Hyperbolic Bends
  
  % Ideal Line Leaf 1: (0 to 1/5)
  \draw[thick, blue] (L1_A) to [bend left=54] node[pos=0.5, auto, black, scale=0.9] {$L_1$} (L1_B);
  
  % Ideal Triangle Leaf 2: (1/4 to 9/20 to 17/20 back to 1/4)
  \draw[thick, blue] (T_A) to [bend left=54] (T_B);
  \draw[thick, blue] (T_B) to [bend left=18] (T_C);
  \draw[thick, blue] (T_C) to [bend left=18] (T_A);
  
  % Centered label for the entire Triangle Leaf L2
  % Puts the label at the gravity center of the 3 boundary points
  \node[black, scale=0.9] at (barycentric cs:T_A=1,T_B=1,T_C=1) {$L_2$};
  
  % Ideal Line Leaf 3: (1/2 to 7/10)
  \draw[thick, blue] (L2_A) to [bend left=54] node[pos=0.5, auto, black, scale=0.9] {$L_3$} (L2_B);

  % 5. Render Red Dots for Leaf Nodes
  \fill[red] (L1_A) circle (0.012) (L1_B) circle (0.012)
              (T_A) circle (0.012) (T_B) circle (0.012) (T_C) circle (0.012)
              (L2_A) circle (0.012) (L2_B) circle (0.012);

  % 6. Render Distinct Dark Dots for Target Points x, y, z
  \fill[darkgray] (Pt_X) circle (0.015)
                  (Pt_Y) circle (0.015)
                  (Pt_Z) circle (0.015);

\end{tikzpicture}
\caption{A degree $5$ primitive major in the $\Pi(1, 1, 2)$ stratum. Since all the endpoints of the leaves are rational numbers this also lives in $RPM(5)$. The pair $(x, y)$ has separation vector $(L_1, L_2)$, the pair $(x, z)$ has separation vector $(L_1, L_2, L_3)$, and the pair $(z, y)$ has separation vector $(L_3, L_2)$.}
\end{figure}

Given a primitive major $M = \{ L_1, ..., L_r \}$ and two points $x, y \in S^1$, we say that a leaf $L_{j}$ separates $x$ and $y$ if they lie in different connected components of $S^1 \setminus \partial L_j$. If $(x, y)$ is an ordered pair of points on $S^1$, we say that $( L_{a_1}, ..., L_{a_k} )$ is the separation vectors of $(x, y)$ if the leaf from $x$ to $y$ passes through $( L_{a_1}, L_{a_2}, ..., L_{a_k})$ exactly in this order. If none of the leaves of $M$ separates $x$ and $y$ then we say their separation vector is $\emptyset$.

\subsection{Thurston's Algorithm}

Fix some $d \geqslant 2$ and let $T: [0, 1) \rightarrow [0, 1)$, $T(x) = dx\mod 1$.

Given an element $M = \{ L_1, ..., L_r \}$ in $PM(d)$, we can consider:
\[
P = \{ T^i(\partial \mathbb{D} \cap L_j )\mid j = 1, 2, ..., r \}.
\]
Define a set $V$ to contain only $\{ x, x \}$ if $P = \{ x \}$ is a singleton set, or otherwise $V = \{ \{ x, y \} \mid x, y \in P, x \neq y \}$. Then we can construct the Thurston graph of $M$, denoted by $G(M)$, to have vertex set $V$. For any element $v = \{ x, y \} \in V$, if $x$ and $y$ belong to a common leaf $L_j$ then there are no edges from $v$. If $\{ x, y \}$ are separated by no leaves in $M$ then there is an edge from $v$ to $w = \{ T(x), T(y) \}$. Otherwise, if the separation vector of $(x, y)$ is $(L_{a_1}, ..., L_{a_s})$ then we have edges to:
\begin{align*}
&\{ T(x), T(L_{a_1}) \}, \\
&\{ T(L_{a_1}), T(L_{a_2}) \}, \\
&..., \\
&\{ T(L_{a_{s-1}}), T(L_{a_s}) \}, \\
&\{ T(L_{a_s}), T(y) \}.
\end{align*}

For a general primitive major $M$, the graph $G(M)$ can have infinitely many vertices. In the special case of a rational primitive major $m \in RPM(d)$, the periodicity of the endpoints of the leaves makes the vertex set finite. Thus we can consider the finite adjacency matrix $A(m)$. By constructions the entries of $A(m)$ are natural numbers $\leqslant d$ (usually just $0$ and $1$) and so the leading eigenvalue $\rho(m)$ of $A(m)$ is a positive real number. Then we can define core entropy $h: RPM(d) \rightarrow \mathbb{R}$ by:
\[
h(m) = \log \rho(A(m)).
\]
We know that $h(m)$ continuously extends to a function on $PM(d)$ by approximating primitive majors by rational primitive majors (see \cite{Ti2} and \cite{GT}), thus giving us a core entropy function $h: PM(d) \rightarrow \mathbb{R}$.

\subsection{Tiozzo's Algorithm}

Here we present an alternative definition of core entropy.

Starting with $M = \{ L_1, ..., L_r \} \in PM(d)$ we define: 
\[
W = \{ \{ y_k(i), y_l(j) \} \mid 1 \leqslant k, l \leqslant r \}.
\]

Here the elements $y_k(i)$ are arbitrary symbols which informally represent the $i^{th}$ iterate of the $k^{th}$ leaf of $M$. The important distinction is that we take the $y_k(i)$ to be pairwise distinct even if the points on $S^1$ they represent are the same. In particular, we can define $\pi(y_k(i))$ to be the point on $S^1$ obtained by iterating the $k^{th}$ leaf $i$ times.

We define the Tiozzo graph $\Gamma(M)$ to have vertex set $W$. For a vertex $v = \{ y_k(i), y_l(j) \}$ we assign edges as follows. If $\pi(y_k(i))$ and $\pi(y_l(j))$ are not separated by any leaf of $M$ then there is one edge to $\{ y_k(i+1), y_l(j+1) \}$.

Otherwise, if the separation vector of the ordered pair $(\pi(y_k(i)), \pi(y_l(j))$ is $(L_{a_1}, ..., L_{a_s})$ then there are edges to:
\begin{align*}
&\{ y_{k}(i+1), y_{a_1}(1) \}, \\
&\{y_{a_1}(1), y_{a_2}(1) \}, \\
&\{y_{a_2}(1), y_{a_3}(1) \}, \\
&..., \\
&\{y_{a_s}(1), y_{l}(k+1)\}.
\end{align*}

Note that the graph $\Gamma(m)$ is always infinite, even if $m \in RPM(d)$.

Given such a graph $\Gamma$, we can denote the number of closed paths of length $n$ by $C(\Gamma, n)$. Then we define the growth rate of $\Gamma$ as:
\[
r(\Gamma) = \limsup_{n \rightarrow \infty} C(\Gamma, n)^{\frac{1}{n}}.
\]

Then according to [GT] we have:

\begin{proposition}
For any $M \in PM(d)$:
\[
h(M) = \log(r(\Gamma(M))).
\]
\end{proposition}

\subsection{Symmetry}

Fix some $d \geqslant 2$. Then it is a trivial but useful fact that the map $T(x) = dx \mod 1$ commutes with the map $S(x) = x + \frac{1}{d-1} \mod 1$. This map $S$ induces a map on the space $PM(d)$, which by abuse of notation we also denote by $S$, by applying $S$ to the endpoints of each leaf of a given primitive major. This map is clearly a homeomorphism and, more interestingly, is entropy preserving.

\begin{lemma}
\label{symlemma}
For any $d \geqslant 2$ and any $M \in PM(d)$, $h(M) = h(S(M))$. 
\end{lemma}

\begin{proof}
Since $S$ and $h$ are both continuous it suffices to prove the claim on $RPM(d)$. Fix some primitive major $m = \{ l_1, ..., l_r \} \in RPM(d)$. Let $G_m$ denote its Thurston Graph and $A_m$ the adjacency matrix of the graph. Define $G_{S(m)}$ and $A_{S(m)}$ similarly.

A vertex $(u, v)$ of the graph $G_{S(m)}$ has the form $(T^n(S(l_i)), T^k(S(l_j)))$. By commutativity, this is the same as $(S(T^n(l_i)), S(T^k(l_j)))$. Hence $S$ induces a map from the vertex set of $G_m$ to the vertex set of $G_{S(m)}$ which, by construction, is a bijection.

The pair $(u, v) = (S(T^n(l_i)), S(T^k(l_j)))$ has separation vector $(S(l_{a_1}), ..., S(l_{a_p}))$ if and only if $(T^n(l_i), T^k(l_j))$ has separation vector $(l_{a_1}, ..., l_{a_p})$. Thus $S$ is edge preserving and so the leading eigenvalues of $A_m$ and $A_{S(m)}$ are the same.
 
\end{proof}

\subsection{Kneading Theory}

Let $M = \{ L_1, ..., L_r \} \in PM(d)$. Then $M$ induces an equivalence relation $\sim$ on $S^1$ by identifying endpoints of each leaf $L_j$. Each leaf $L_j$ is shrunk down to a single point giving a deformed circle $S^1 / \sim$. The connected components of $S^1 / \sim $ minus the (projection of) the endpoints of the leaves induces a partition of $S^1$ into $d$ pieces $I_1, ..., I_d$ each of length $\frac{1}{d}$.

Given a point $x \in S^1 \setminus \{ L_1, ..., L_r \}$, we can define $\sigma_M(x) = j$ if $x \in I_j$. Then for each $k$, we can define $\sigma_{M, k}(x) = \sigma_{M}(T^k(x))$. This definition makes sense provided that the iterates of $x$ never land on a leaf of $M$. Thus we can define the kneading sequence:
\[
\Sigma_M(x) = (\sigma_{M, 1}(x), \sigma_{M, 2}(x), ... ).
\]
Since the endpoints of the leaves of each $M = \{ L_1, ..., L_r \} \in PM(d)$ are identified by $T$, we can further define $\Sigma_{M}(L_j)$ for each $j$ by taking the kneading sequence of any endpoint of $L_j$. As before this only makes sense if the iterates of the leaves never land on a leaf.

\section{Total Separation}

For each pair of iterates of the leaves of a primitive major, the number of outgoing edges on the corresponding Thurston (or Tiozzo) graph equals the number of leaves which the pair separates (plus one). Thus, at least informally, the more leaves being separated the more entropy the primitive majors accumulates. We make this observation precise with the concept of a vertex being totally separated.

\begin{definition}
Let $M = \{ L_1, ..., L_r \} \in PM(d)$ and let $G(M)$ be the corresponding Thurston graph. We say that a vertex $(x, y)$ of $G(m)$ is totally separated if:
\begin{enumerate}
	\item the separation vector of $(x, y)$ is $\{ L_1, ..., L_r \}$, and
	\item if there is a path from $(x, y)$ to $(z, w)$ then the separation vector of $(z, w)$ is also $\{ L_1, ..., L_r \}$.
\end{enumerate}
We say that the graph $G(M)$ is totally separated if it contains a totally separated vertex and that the primitive major $M$ is totally separated if $G(M)$ is. Note we may similarily define totally separated for the Tiozzo graph $\Gamma(M)$.
\end{definition}

When $M = \{ L_1, ..., L_r \}$ is totally separated we can impose a preferred ordering on the leaves. Indeed, if a vertex $(u, v)$ is totally separated then the separation vector is $\{ L_1, ..., L_r \}$ meaning that, after possibly relabeling the leaves, we can assume that going from $u$ to $v$ crosses $(L_1, ..., L_r)$. Because of this we obtain:

\begin{lemma}
\label{totseppairs}
If $M = \{ L_1, ..., L_r \} \in PM(d)$ is totally separated and $r \geqslant 2$ then (after possibly relabeling the leaves) for each $j$, the pair $(T(L_j), T(L_{j+1}))$ is totally separated. Moreover, any totally separated vertex $v$ has an edge in $G(M)$ to $(T(L_j), T(L_{j+1}))$. If $r = 1$ then the pair $(T(L_1), T^2(L_1))$ is totally separated. Moreover any totally separated vertex $v$ has a path of length $2$ in $G(M)$ to $(T(L_1), T^2(L_1))$.
\end{lemma}

\begin{proof}
Assume $r \geqslant 2$ and suppose $G(M)$ is totally separated. Then there is a vertex $(x, y)$ in $G(M)$ which is totally separated. After reordering we can take its separation vector to be $(L_1, ..., L_r)$. Thus there are edges:
\begin{align*}
&(x, y) \rightarrow (T(x), T(L_1)), \\
&(x,y) \rightarrow (T(L_1), T(L_2)), \\
&(x, y) \rightarrow (T(L_2), T(L_3)), \\
&..., \\
&(x, y) \rightarrow (T(L_{r-1}), T(L_r)) \\
&(x, y) \rightarrow (T(y), T(L_{r})). 
\end{align*}
By definition, each of these is totally separated which proves the claim.

When $r = 1$ similar reasoning to the above shows that a totally separated pair $(x, y)$ will have an edge to $(T(L_1), T(x))$. This must also then be totally separated and it has an edge to $(T(L_1), T^2(L_1))$ which proves the claim.
\end{proof}

\begin{proposition}
\label{definedkneading}
Let $M = \{ L_1, ..., L_r \} \in PM(d)$ such that $G(M)$ is totally separated. Then $\Sigma_M(L_j)$ is well defined for each $j$.
\end{proposition}

\begin{proof}
Assume $r \geqslant 2$ and suppose $G(M)$ is totally separated for some $M = \{ L_1, ..., L_r \} \in PM(d)$. Then, after relabeling, know that each pair $(T(L_j), T(L_{j+1}))$ is totally separated and hence has separation vector $(L_1, L_2, ..., L_{r})$ or $(L_r, ..., L_2, L_1)$ implying $T(L_j)$ does not lie on a leaf. One of the outgoing edges from this vertex will be $(T(L_1), T^2(L_j))$ or $(T(L_r), T^2(L_j))$ which then must also be totally separated. Thus by the same reasoning $T^2(L_j)$ can not lie on a leaf either. Repeating we see every iterate of $L_j$ do not land on any leaves and so $\Sigma_M(L_j)$ is well defined for each $j$.

If $r=1$ then a similar argument shows that, since $(T(L_1), T^2(L_1))$ is totally separated, that $(T(L_1), T^n(L_1))$ is totally separated for all $n \geqslant 2$ and hence all the iterates of $L_1$ do not land on $L_1$. 
\end{proof}

Our aim is to show that maximal entropy occurs only when the graph $G(M)$ is totally separated. We begin by considering $RPM(d)$. In this case the corresponding Thurston graph will be finite and so the result is fairly straightforward linear algebra. This proof is essentially a part of the proof of Theorem D in \cite{AF} rewritten in the language of laminations.

\begin{lemma}
\label{RPMlemma}
Suppose $m = \{ l_1, ..., l_r \} \in RPM(d)$. Then the core entropy $h(m)$ is $\leqslant \log(r+1)$ with equality if and only if there is some pair $(x, y)$ which is totally separated.\end{lemma}

\begin{proof}
Consider $A$ the $c \times c$ transition matrix associated to the Thurston graph $G(m)$ of $m$. Since $m$ is rational, $G(m)$ and $A$ are both finite. Denote by $r(A)$ the leading eigenvalue of $A$. If $m$ consists of $r$ leaves then each vertex in $G(m)$ has at most $r+1$ outgoing edges, hence the sum of the entries in each column of $A$ is at most $r+1$. This forces that $r(A) \leqslant r+1$ and hence that $h(m) = \log(r(A)) \leqslant \log(r+1)$.

Suppose that $G(m)$ does not contain a totally separated edge. Then for any pair $(x, y)$, any path from $(x, y)$ eventually leads to an edge with strictly fewer then $r+1$ outgoing edges.

It follows that for some $n$ large enough, the sum of the entries in any column of $A^n$ are all strictly less then $(r+1)^n$. Let $s < r+1$ so that the sum of the entires of any column of $M^n$ is $\leqslant s^n$.

By abuse of notation we will say that $u \leqslant v$ for two vectors $u,  v$ if for all $j$, the $j^{th}$ entry of $u$ is $\leqslant$ the $j^{th}$ entry of $v$. Then if we let $v$ denote the vector containing all $1's$ of dimension $c$ the same as the dimension of the columns of $M$ then we have:
\[
vA^n \leqslant s^n v.
\]
Repeating we find that for all $p \in \mathbb{N}$:
\[
vA^{np} \leqslant s^{np}v.
\]
This implies that:
\[
vA^{np}v^{T} \leqslant cs^{np}.
\]
Then:
\begin{align*}
r(A) &= \lim_{p \rightarrow \infty} (vA^p v^{T})^{\frac{1}{p}} \\
&= \lim_{p \rightarrow \infty} (vA^{pn}v^{T})^{\frac{1}{np}} \\
&\leqslant \lim_{p \rightarrow \infty} (cs^{np})^{\frac{1}{np}} \\
&= s \\
&< r+1.
\end{align*}

\end{proof}

The above only works if the primitive major is rational. Here we try to generalize the main result from before to arbitrary primitive majors.

\begin{lemma}
\label{totsep1}
Suppose that $M = \{ L_1, L_2, ..., L_r \} \in PM(d)$. If $h(M) = \log(r+1)$ then M is totally separated.
\end{lemma}

\begin{proof}
Suppose that $M$ is not totally separated and that $r > 1$. Order the leaves of $M$ so that there is some pair of points on $S^1$ with separation vector $(L_1, L_2, ..., L_r)$ (if this is not possible then we would have $h(M) \leqslant \log(r) < \log(r+1)$). Suppose further that no iterate of any leaf lands on another leaf. Since $M$ is not totally separated, there is some number $N(M)$ such that the number of paths of length $N(M)$ starting from the vertex $(T(L_1), T(L_2))$ in the Thurston graph $G(M)$ is strictly less then $(r+1)^{N(M)}$, call the number of paths $P(M)$. By continuity, it follows if $m = \{ l_1, ..., l_r \} \in RPM(d)$ is close enough to $M$ then the number of paths of length $N(M)$ starting from $(T(l_1), T(l_2))$ in $G(m)$ is also $P(M)$.

Let $(u, v)$ be any vertex in $G(m)$. If the number of edges from $(u, v)$ is $< r+1$ then the number of paths of length $N(M)+1$ starting from $(u, v)$ is at most $r(r+1)^{N(M)}$. If the number of edges is $r+1$ then there is an edge from $(u, v)$ to $(T(L_1), T(L_2))$. Thus there are at most $r(r+1)^{N(M)} + P(M)$ paths of length $N(M)+1$ starting from $(u, v)$. As $P(M) < (r+1)^{N(M)}$, this upper bound is also $< (r+1)^{N(M)+1}$.

Thus for all $m$ close enough to $M$, and any vertex $(u, v)$ in $G(m)$, the number of paths of length $N(M)+1$ starting from $(u, v)$ is at most $r(r+1)^{N(M)} + P(M)$ and hence is uniformly less then $(r+1)^{N(M)+1}$. Let $\gamma$ be such that $(r+1)^{N(M)} + P(M) < \gamma^{N(M)+1} < (r+1)^{N(M)+1}$. Then by the proof of Lemma \ref{RPMlemma} we have $h(m) \leqslant \log(\gamma) < \log(r+1)$ for all $m$ close to $M$. Hence $h(M) \leqslant \log(\gamma) < \log(r+1)$ also.

Still assuming $r > 1$ suppose that some iterate of a leaf maps onto the endpoint of a leaf. This divides into two further cases, either a leaf maps to itself or to a distinct leaf. Suppose that some leaf, say $L_i$, maps onto one of its own endpoints, that is, $T^n(L_i) \in \partial L_i$. Then any separation vector of any vertex in $G(M)$ including $T^n(L_i)$ must have cardinality less then $r$, hence $N(M) \leqslant n + 1$. Moreover, notice that this forces the endpoints of $L_i$ to be rational numbers. Consider $m \in RPM(d)$ belonging to the same stratum as $M$ and so that $L_i \in m$. Even if $m$ is arbitrarily close to $M$ it may be that $N(m) > N(M)$ however notice that $N(m) \leqslant n+1$ also. Indeed, if there are $(r+1)^{n+1}$ paths of length $n+1$ coming from a vertex $(T(L_1), T(L_2))$ in $G(m)$ then one of these paths terminates at a vertex containing $T^n(L_i)$. If we take $\gamma$ such that $\gamma^{n+1} < (r+1)^{n+1}$ then $h(m) \leqslant \log(\gamma)$ and so by continuity $h(M) \leqslant \log(\gamma) < \log(r+1)$ also.

Now assume that some leaf iterates onto a different leaf, that is, that $T^n(L_i) \in L_j$ for some $n$ and some $i \neq j$. Then if we perturb $L_i$ slightly to some new leaf of the same degree with rational endpoints $l_i$ then $T^n(l_i)$ will be a rational number close to $L_j$. Perturbing $L_j$ to $l_j$ so that $T^n(l_i)$ is an endpoint of $l_j$ we can then perturb all the other leafs to rational numbers to construct rational primitive majors arbitrarily close to $M$ such that $T^n(l_i)$ lands on an endpoint of $l_j$. Then by the same reasoning as the previous case we have $h(m)$ is uniformly less than $\log(r+1)$ for these rational primitive majors, hence by continuity $h(M)$ is also strictly less than $\log(r+1)$.

Finally consider the case where $r=1$. If the unique leaf  iterates onto itself then its endpoints are rational, hence we are in $RPM(d)$ and the result follows from Lemma \ref{RPMlemma}. Otherwise, notice that any vertex $(u, v)$ with $4$ paths of length $2$ starting from $(u, v)$ will have a path to $(T(L_1), T^2(L_1))$. Hence the same argument as in the first case applied to this vertex gives the result.

\end{proof}

Note we only proved one direction of the if and only if. To obtain the other direction for a rational primitive major is trivial. In theory we could approximate a general element of $PM(d)$ by elements of $RPM(d)$ like above and try to prove the converse that way however the estimates become complicated. Instead we leverage Tiozzo's algorithm which provides a much easier proof.

\begin{lemma}
\label{totsep2}
Suppose that $M = \{ L_1, L_2, ..., L_r \} \in PM(d)$. If $M$ is totally separated then $h(M) = \log(r+1)$.
\end{lemma}

\begin{proof}
Let $\Gamma$ be the Tiozzo graph of $M$. First suppose that $r > 1$. Since $M$ is totally separated we know by lemma \ref{totseppairs} that the vertex $v = (y_1(1), y_2(1))$ is totally separated. Thus, for any $n$, there are $(r+1)^n$ paths of length $n$ starting from $v$. 

Let $w = (x, y)$ be the endpoint of such a path. Then since $w$ is also totally separated there are edges:
\begin{align*}
&w \rightarrow (y_1(1), y_2(1)), \\
&w \rightarrow (y_2(1), y_3(1)), \\
&..., \\
&w \rightarrow (y_{r-1}(1), y_r(1)). 
\end{align*}

Thus there is a path of length $n+1$ from $v$ back to itself. This works for any path of length $n$ described above and so we have at least $(r+1)^n$ paths of length $n+1$ from $v$ to $v$. Thus:
\[
C(\Gamma(M), n+1) \geqslant (r+1)^n
\]
and so:
\[
r(\Gamma(M)) = \limsup_{n \rightarrow \infty} C(\Gamma(M), n+1)^{\frac{1}{n+1}} \geqslant r+1.
\]

This makes $h(M) \geqslant \log(r+1)$ and since $h(M) \leqslant \log(r+1)$ is always true we have equality.

Now suppose $r=1$. Then by lemma \ref{totseppairs} we know $(y_1(1), y_1(2))$ is totally separated. By definition this means there are $2^n$ paths of length $n$ starting from it. Let $w$ be the endpoint of such a path. As $w$ is also totally separated, there is a path of length $2$ from $w$ to $(y_1(1), y_1(2))$. Thus there are at least $2^n$ paths of length $n+2$ from $(y_1(1), y_1(2))$ back to itself.

From here, the same argument as in the $r > 1$ case gives $h(M) \geqslant \log(2)$. Since there is only one leaf, we have $h(M) \leqslant \log(2)$ as well so we are done.
\end{proof}

Finally, by combining Lemmas \ref{totsep1} and \ref{totsep2} we obtain:
\begin{theorem}
\label{totsepthm}
Given $M = \{ L_1, L_2, ..., L_r \} \in PM(d)$, $h(M) = \log(r+1)$ if and only if $M$ is totally separated. 
\end{theorem}

\section{Maxima over $PM(d)$}

Here our goal is to classify all $M \in PM(d)$ such that $h(M) = \log(d)$. First, notice that for $h(M) = \log(d)$ it must be the case that $M$ consists of $d-1$ leaves so all maxima occur in the stratum $\Pi(1, 1, ..., 1)$. In this case, there can only be two intervals in the partition of $S^1$ induced by $M$, denote them by $I_1$ and $I_2$, such that for $(x, y) \in I_1 \times I_2$, we have the separation vector is $\{ L_1, ..., L_{d-1} \}$ in some order. Reordering if necessary, we can assume the separation vector is precisely $(L_1, L_2, ..., L_{d-1})$ or $(L_{d-1}, L_{d-2}, ..., L_1)$. It follows from Lemma \ref{totseppairs} that for each $j$, $(T(L_j), T(L_{j+1}))$ (or just $(T(L_1), T^2(L_1))$ when $d=2$) is totally separated and hence each $T(L_j)$ belongs to one of $I_1, I_2$. Then we will choose to label $I_1$ as the interval for which $T(L_1)$ lives in and $I_2$ the other. 

Interestingly, the classification depends on whether $d$ is even or odd. In both cases we will show there are exactly $d-1$ primitive majors which maximize core entropy. In the even case all these maxima are identified under $S(z) = z + \frac{1}{d-1} \mod 1$ while in the odd case there are two distinct equivalence classes.

% even case
\begin{proposition}
\label{evenprop}
Let $d$ be even and suppose $M \in PM(d)$ such that $h(M) = \log(d)$. Write $M = \{ L_1, ..., L_{d-1} \}$ as above and label the partition of $S^1$ it induces as above. Then $\Sigma_{M}(L_{2j+1}) = (1, \overline{2})$ and $\Sigma_{M}(L_{2j}) = (\overline{2})$.
\end{proposition}

\begin{proof}
If $M$ is totally separated then we know by Lemma \ref{totseppairs} each consecutive pair $(T(L_j), T(L_{j+1}))$ is totally separated. Since there are $d-1$ leaves there are only two components which these iterates can live in, call them $I_1$ and $I_2$. Moreover, we have that $T(L_{2j})$ all lie in the same interval while $T(L_{2j+1})$ lie in the other. Denote by $I_1$ the interval which $T(L_1)$ (and hence $T(L_{2j+1})$ lie in for all $j$) and $I_2$ the other. Then we have that $\sigma_{M, 1}(L_{2j+1}) = 1$ and $\sigma_{M, 1}(L_{2j}) = 2$. 

Given a pair $(x, y) \in I_1 \times I_2$ which is totally separated, we have that the separation vector is either $(L_1, L_2, ..., L_{d-1})$ or $(L_{d-1}, L_{d-1}, ..., L_1)$. Hence we have two edges, either to $(T(x), T(L_1))$ and $(T(y), T(L_{d-1}))$ or to $(T(x), T(L_{d-1}))$ and $(T(y), T(L_1))$. Since $T(L_1)$ and $T(L_{d-1})$ both lie in $I_1$, regardless of which case we are in we see $T(x)$ and $T(y)$ must live in $I_2$. Repeating, we find that all iterates $T^n(x), T^n(y)$ are in $I_2$. In particular, this gives $T^n(L_{j}) \in I_2$ for all $n \geqslant 2$ and for all $j$. This means $\sigma_{M, n}(L_j) = 2$ for all $n \geqslant 2$ and for all $j$ which completes the proof.
\end{proof}

\begin{lemma}
\label{kneading-lemma-1}
Suppose that $M = \{ L_1, ..., L_{d-1} \} \in PM(d)$. Suppose that $x \in I_k$ has kneading sequence $(\overline{k})$ and that $T(\partial I_k) \cap I_k = \emptyset$ (for $k = 1$ or $2$). Then $T(x) = x$.
\end{lemma}

\begin{proof}
Define $A_0 = \{a \in I_k \mid T(a) \in I_k \}$. Since $len(I_k) = \frac{1}{d}$ we have $len(A_0) = \frac{1}{d^2}$. Also, since $I_k$ is a connected interval, $A_0$ is either two intervals or one. As $T(\partial I_k) \cap I_k = \emptyset$ the former cannot happen and so $A_0$ is a connected interval.

Inductively $A_{j+1} = \{ a \in A_j \mid T(a) \in A_j \}$. By the same reasoning as above, $A_{j+1}$ is a connected interval and $len(A_{j+1}) = \frac{1}{d^{j+2}}$. Thus by the usual nested interval argument, we have $A_\infty := \bigcap_{j} \overline{A_j}$ is a singleton, and hence $A_\infty = \{ x \}$.

Since $T(x)$ also has kneading sequence $(\overline{k})$ and there is only one such element in $I_k$, it follows that $T(x) = x$ as claimed.
\end{proof}

\begin{theorem}
\label{eventhm}
Let $d$ be even. There are $d-1$ elements in $PM(d)$ which maximize core entropy. All are equivalent under $S(z) = z + \frac{1}{d-1} \mod 1$.

% $M_{max} = \{ (\frac{k}{d}, \frac{d-k}{d}) \}_{k=1}^{d/2-1} \cup \{ (\frac{k}{d} + \frac{1}{2d}, \frac{d-k}{d} - \frac{1}{2d}) \}_{k=0}^{d/2-1}$

\end{theorem}

\begin{proof}
Let $M = \{ L_1, ..., L_{d-1} \} \in PM(d)$ such that $h(M) = \log(d)$. Then by Proposition \ref{evenprop}, we know that $\Sigma_{M}(L_{2j+1}) = (1, \overline{2})$ and that $\Sigma_M(L_{2j}) = (\overline{2})$. Then by Lemma \ref{kneading-lemma-1} we know that $T(L_{2j}) = p$ for some fixed point $p$. Moreover, since $\Sigma_M(T(L_{2j+1})) = (1, \overline{2})$ it follows $T(L_{2j+1})$ is also a fixed point of $T$. As it belongs to the same interval as $p$ and each interval contains at most $1$ fixed point, we have $p = T^2(L_{2j+1})$. Also, as $T(L_{2j+1}) \in I_1$ we have each $T(L_{2j+1})$ is the same preimage of $p$ for all $j$.

In order for $T(L_{2j})$ to be totally separated from $T(L_{2j+1})$ we must have that exactly half of the endpoints of the leaves $L_1, ..., L_{d-1}$ lie in one component of $S^1 \setminus \{ T(L_{2j}), T(L_{2j+1}) \}$ and that the other half lie in the other component. There are exactly $2(d-1)$ endpoints in total, all of which are preimages of $T(L_{2j})$ or preimages of $T(L_{2j+1})$. Notice that $T(L_{2j+1})$ and $T(L_{2j})$ are both preimages of $p$ neither of which is an endpoint. Thus there are $d-2$ preimages of $T(L_{2j+1})$ to use for endpoints and $d$ preimages of $T(L_{2j+1})$ to use for preimages totally to $2d-2$. Hence each preimage must be used as an endpoint. 

As preimages of $T$ are all equidistributed along $S^1$, the only way to choose $T(L_{2j})$ so that the endpoints can be made into a lamination with $d-1$ leaves which all separate $T(L_{2j})$ from $T(L_{2j+1})$ is if $T(L_{2j+1}) = p + \frac{1}{2}$. Thus each such lamination $M$ is determined uniquely by a choice of fixed point $p$. There are $d-1$ fixed point under multiplication by $d$ and hence $d-1$ maxima. And clearly all laminations described this way are identified under $S(z)$ hence there is only one equivalence class.
\end{proof}

Now we move on to the odd case. The proof strategy is essentially the same as before however the details become more technical. Instead of having one type of kneading sequence for each leaf it turns out there are two cases, one will be the same as the even case and correspond to fixed points of $T$ while the other will correspond to points of period $2$.

% odd case
\begin{proposition}
\label{oddprop}
Let $d$ be odd and suppose $M \in PM(d)$ such that $h(M) = \log(d)$. Write $M = \{ L_1, ..., L_{d-1} \}$ as above. Then either $\Sigma_{M}(L_{2j+1}) = (\overline{1})$ and $\Sigma_{M}(L_{2j}) = (\overline{2})$ or $\Sigma_{M}(L_{2j+1}) = (\overline{1, 2})$ and $\Sigma_{M}(L_{2j}) = (\overline{2, 1})$.
\end{proposition}

\begin{proof}
As in the even case, if $M$ is totally separated then by Lemma \ref{totseppairs} each consecutive pair $(T(L_j), T(L_{j+1}))$ is totally separated. Again we denote by $I_1$ the unique interval for which each $T(L_{2j+1})$ lives in and $I_2$ the interval for which $T(L_{2j})$ lies.

Give  a pair $(x, y) \in I_1 \times I_2$ there are two cases for the separation vector. Since $T(L_1) \in I_1$ and $T(L_{d-1}) \in I_2$ these cases are no longer equal as in the $d$ even case and must be handled separately.

Case 1: Suppose the separation vector of $(x, y) = (L_1, L_2, ..., L_{d-1})$. Then we have edges to $(T(x), T(L_1))$ and to $(T(y), T(L_{d-1})$. For these to be totally separated we must have $T(x) \in I_2$ and $T(y) \in I_1$. Thus if $(x, y) \in I_1 \times I_2$ then $(T(x), T(y)) \in I_2 \times I_1$. Repeating, we see $T^{2n}(x), T^{2n+1}(y) \in I_2$ and $T^{2n+1}(x), T^{2n}(y) \in I_1$. This implies $\Sigma_M(L_{2j}) = (\overline{2, 1})$ and $\Sigma_M(L_{2j+1}) = (\overline{1, 2})$.

Case 2. Suppose the separation vector of $(x, y) = (L_{d-1}, L_{d-2}, ..., L_1)$. Then we have edges to $(T(x), T(L_{d-1})$ and $(T(y), T(L_{1})$. Then for these to be totally separated we must have $T(x) \in I_1)$ and $T(y) \in I_2$. Thus, if $(x, y) \in I_1 \times I_2$ then $(T(x), T(y)) \in I_1 \times I_2$ also. Repeating, we see $T^n(x) \in I_1$ and $T^n(y) \in I_2$. This implies $\Sigma_M(L_{2j}) = (\overline{2})$ and $\Sigma_M(L_{2j+1}) = (\overline{1})$.
\end{proof}

\begin{lemma}
\label{kneading-lemma-2}
Suppose that $M = ( L_1, ..., L_{d-1} ) \in PM(d)$ with $P = \{ I_1, ..., I_{d-1} \}$ the induced partition. Suppose that $x \in I_1$ has kneading sequence $(\overline{1, 2})$ and that $T(\partial I_1) \in I_1$, $T(\partial I_2) \in I_2$. Then $T^2(x) = x$.
\end{lemma}

\begin{proof}
Consider $P_1 = \{ z \in I_2 \mid T(z) \in I_1 \}$. Since $T(\partial I_2) \in I_2$ this set is an interval of length $\leqslant \frac{1}{d^2}$. Define $P_2 = \{ z \in I_1 \mid T(z) \in P_1 \}$. Then since $T(I_1) \in I_1$ this is an interval of length $\leqslant \frac{1}{d^3}$. Repeating we construct sets $P_j$ so that the even and odd subsequences are nested subsets of $I_2$ and $I_1$ respectively. Then the intersection of the even and odds are singletons which, by construction, $T$ maps to each other. Since $x$ must be in the intersection of the $P_{2j+1}$ this proves the claim.
\end{proof}

The above shows that the odd degree maxima come in two types. On consisting of two fixed points and the other of two period two points. The analysis of the fixed points is similar to the even case, the only difference being a single primitive major will require two fixed points rather than one fixed point and its preimage. The period two case is more complicated. If we include fixed points then there are $d^2-1$ points of period two under multiplication by $d$. Most points are not suitable for constructing primitive majors since we will want the point and its image to be totally separated. That is, we need a pair of period two points $x$ and $y = T(x)$ so that exactly half of the preimages of $x$ and $y$ (excluding $x$ and $y$) lie in each component of $S^1 \setminus \{ x, y \}$. Thus we classify all such pairs:

\begin{lemma}
Suppose $d \geqslant 3$ is odd. Let $x \in (0, \frac{1}{d})$ have period two and define $y = T(x)$. Suppose that exactly half of the preimages of $x$ and $y$ (excluding $x$ and $y$) lie in each component of $S^1 \setminus \{ x, y \}$. Then $x = \frac{n+1}{d^2-1} = \frac{1}{2(d-1)}$ where $d = 2n + 1$. Moreover, we have $y - x = \frac{1}{2}$.
\end{lemma}

\begin{proof}
Recall that points of period two (including fixed points) are of the form $\frac{k}{d^2-1}$. Thus we may assume $x$ has this form and need only show $k = n+1$.

As $x < \frac{1}{d}$ we know $y = T(x) = dx$. Then the lengths of the two components of $S^1 \setminus \{ x, y \}$ are $dx - x$ and $1 - (dx - x)$ respectively. 

Since $y$ is a preimage of $x$ we know that each distance $\frac{1}{d}$ we go from $y$ we will reach another preimage. Thus to have $\frac{d-1}{2}$ preimages of $y$ in each component, we know the length of each component is at most $\frac{\frac{d-1}{2} + 1}{d}$. Thus:

\begin{align*}
dx -x &\leqslant \frac{\frac{d-1}{2} + 1}{d}.
\end{align*}
Substituting in $x = \frac{k}{d^2-1}$ gives:
\begin{align*}
\frac{(d-1)k}{d^2-1} &\leqslant \frac{1}{2} + \frac{1}{2d},
\end{align*}
and hence:
\begin{align*}
k &\leqslant \frac{d+1}{2} + \frac{1}{2} + \frac{1}{2d} \\
&= n + 1 + \frac{1}{2} + \frac{1}{4n+2}.
\end{align*}

Similarly we have:
\begin{align*}
1 - (dx-x) &\leqslant \frac{\frac{d-1}{2} + 1}{d}.
\end{align*}
So substituting $x = \frac{k}{d^2-1}$ gives:
\begin{align*}
1 - \frac{1}{2} - \frac{1}{2d} &\leqslant \frac{(d-1)k}{d^2-1}.
\end{align*}
Rearanging for $k$ gives:
\begin{align*}
k &\geqslant \frac{d+1}{2} - \frac{1}{2} - \frac{1}{2d} \\
&= n + \frac{1}{2} - \frac{1}{4n+2}.
\end{align*}

Putting these together we have:

\[
n + \frac{1}{2} - \frac{1}{4n+2} \leqslant k \leqslant n + 1 + \frac{1}{2} + \frac{1}{4n+2}.
\]

Notice that as $d = 2n+1 \geqslant 3$ we have $n \geqslant 1$. Thus, since $k$ must be an integer, this implies $k = n+1$ as claimed.
\end{proof}

\begin{theorem}
\label{oddthm}
Let $d$ be odd. There are $d-1$ laminations in $PM(d)$ which maximize core entropy. The maxima fall into two distinct equivalence classes under $S(z) = z + \frac{1}{d-1} \mod 1$.
\end{theorem}

\begin{proof}
Let $M = \{ L_1, ..., L_{d-1} \}$ have $h(M) = \log(d)$. Then by Proposition \ref{oddprop} there are two cases:

Case 1.  Suppose that $\Sigma_{M}(L_{2j}) = (\overline{2})$ and $\Sigma_{M}(L_{2j+1}) = (\overline{1})$.  Then by Lemma \ref{kneading-lemma-1} we know that there are fixed point $x \in I_1, y \in I_2$ so that $T(L_{2j}) = y$ and $T(L_{2j+1}) = x$. Each of these fixed points has $d-1$ preimages other then themselves and as the leaves $L_j$ consist of $2d-2$ endpoints on $S^1$, this means the endpoints of $L_{2j}$ are exactly the preimages of $y$ (except $y$) and the endpoints of $L_{2j+1}$ are exactly the preimages of $x$ (except $x$). As the endpoints are equidistributed on $S^1$, the only way that they can be connected to make a primitive major with $d-1$ leaves which separate $x$ from $y$ is if $x = y+\frac{1}{2}$. Thus the pairs of fixed points $\{ x, x+\frac{1}{2} \}$ each determine a maximum primitive major. Each maximum can be reached from another by successively adding $\frac{1}{d-1}$ to the leaves.

Case 2. Suppose that $\Sigma_{M}(L_{2j}) = (\overline{2, 1})$ and that $\Sigma_{M}(L_{2j+1}) = (\overline{1, 2})$. Then by Lemma \ref{kneading-lemma-2} we know that there are elements $x \in I_1$, $y \in I_2$ such that $T(x) = y, T(y) = x, T(L_{2j}) = y, T(L_{2j+1}) = x$. Notice that each of $x$ and $y$ has $d-1$ preimages excluding $x$ and $y$. Each of these must be an endpoint of the leaves $L_j$ and there are $2(d-1)$ such endpoints, hence the preimages equal the set of endpoints.

Suppose further that $x \in (0, \frac{1}{d})$. Since $x$ and $y$ are separated by all the leaves $L_j$, we know from the previous lemma that $x = \frac{1}{2(d-1)}$ and $y = x + \frac{1}{2}$.

If $x$ is not in $(0, \frac{1}{d})$ then we can add $\frac{1}{d-1}$ some number of times to the leaves of $M$ to get such a primitive major. Hence we have $x = \frac{1}{2(d-1)} + \frac{k}{d-1}$ and $y = x + \frac{1}{2}$.

Finally, note that in case 1 there are $d-1$ fixed points of $T$ and so $\frac{d-1}{2}$ maxima. In case 2, there was one maxima assuming $x \in (0, 1/d)$. We can apply $S$ $\frac{d-1}{2}$ times to produce the rest of the maxima, hence there are $\frac{d-1}{2}$ maxima here. Summing we obtain $d-1$ possible maxima.
\end{proof}

Finally, by combining Theorems \ref{eventhm} and \ref{oddthm} we obtain Theorem A.

\section{Maxima on the Unicritical Stratum}

Here we restrict our focus to the unicritical stratum of $PM(d)$ for $d \geqslant 3$. Our main result is that the set of maxima is a Cantor set.

Recall that $\Pi(d)$, the unicritical stratum of $PM(d)$, consists of one leaf $M_a = \{ (a, a + 1/d, a+2/d, ..., a + (d-1)/d) \}$. This allows us to study $\Pi(d)$ with only $a \in [0, 1/d]$. By Lemma \ref{symlemma} we have that $h(M_{a}) = h(M_{a + \frac{1}{d-1}})$. As $a$ runs through $[0, \frac{1}{d(d-1)}]$, $a + \frac{d-1}{d} + \frac{1}{d-1}$ will run through $[\frac{1}{d(d-1)}, \frac{2}{d(d-1)}]$. Thus, by repeatedly adding $\frac{1}{d-1}$ we can relate any $M_a$ with $a \in [0, \frac{1}{d(d-1)}]$ to the rest of the primitive majors in $[\frac{1}{d(d-1)}, \frac{1}{d}]$. Therefore we can restrict further to $M_a$ for $a \in [0, \frac{1}{d(d-1)}]$.

\begin{lemma}
\label{cantorkneading}
$h(M_a) = \log(2)$ if and only if $\Sigma_{M_a}(a) \in \{ A_1 \} \times \{ A_2, ..., A_d \}^{\mathbb{N}}$ for some distinct choices of $A_j \in \{0, 1,  ..., d-1 \}$.
\end{lemma}

\begin{proof}
First suppose $h(M_a) = \log(2)$ and let $G$ be the Thurston graph of $M_a$. Since the entropy is maximal $G$ must have some totally separated vertex $v = (T^i(a), T^j(a))$. Since $v$ is totally separated it has an edge to $(T(a), T^{i+1}(a))$. This is also totally separated and so has an edge to $(T(a), T^2(a))$. Repeating this logic we have a path from $v$ to all vertices of the form $(T(a), T^n(a))$ for $n \geqslant 2$. The only way for these to be totally separated is if $T(a)$ lives in a different connected component of $S^1 \setminus \{ a, a+1/d, ..., a+(d-1)/d \}$ then all further iterates of $a$. Thus $\Sigma_{M_a}(a)$ has the desired form.

Conversely, suppose $\Sigma_{M_a}(a)$ is as above for some $a$. Then, as noted above, $T(a)$ lives in a different connected component of $S^1 \setminus \{ a, a+1/d, ..., a+(d-1)/d \}$ then all further iterates of $a$. Consider the vertex $(T(a), T^2(a))$. As $T(a)$ and $T^2(a)$ are separated, there will be two edges, one to $(T(a), T^2(a))$ and the other to $(T(a), T^3(a))$. Repeating we see there are edges only to pairs of the form $(T(a), T^n(a))$ for $n \geqslant 2$. Thus we have that $(T(a), T^2(a))$ is totally separated and so $h(M_a) = \log(2)$.
\end{proof}

\begin{remark}
Notice that since $a \in [0, \frac{1}{d(d-1)})$ we will have $T(a) \in [a, \frac{1}{d-1}) \subseteq [a, a + \frac{1}{d})$. Thus in the above lemma we can take $A_1 = 0$.
\end{remark}

The above lemma is highly suggestive that the set of maxima for $h$ restricted to $\Pi(d)$ should be a Cantor set. Our difficulty is two fold. First, it is not obvious that each kneading sequence in the symbol space is actually realized as the kneading sequence of a primitive major. Secondly, it is not clear that the sequences should uniquely determine primitive majors.  In fact both of these things are true, as we now prove.

To begin our construction, consider the primitive major:
\[
M_0 = \left \{ \left (0, \frac{1}{d}, \frac{2}{d}, ..., \frac{d-1}{d} \right ) \right \}.
\]
We can label the partition of $S^1$ induced by $M_0$ by $I_0 = (0, \frac{1}{d})$, $I_1 = (\frac{1}{d}, \frac{2}{d})$, and so on. Then we can consider the Cantor set $C$ defined by taking all elements of $[0, 1]$ who admit a base $d$ representation with no $0$.

Notice that for any $z \in C$, we either have $T^n(z) = \frac{k}{d}$ for some $n, k \in \mathbb{N}$ or else the kneading sequence of $z$ with respect to $M_0$ is well defined and $\Sigma_{M_0}(z)$ contains no $0$. In fact, the kneading sequence is the base $d$ representation.

Now consider:
\[
C' = \{ \frac{z}{d^2} \mid z \in C \}.
\]

Since $C$ is a Cantor set it follows that $C'$ is also.

\begin{theorem}
\label{B1}
If $z \in C'$ then $h(M_z) = \log(2)$.
\end{theorem}

\begin{proof}
Fix $z \in C'$. First we show that $\Sigma_{M_z}(z)$ is well defined. Indeed, notice that any element in $(\frac{k}{d}, \frac{k}{d} + \frac{1}{d^2}]$ maps to $(0, \frac{1}{d}]$ under iteration by $T$. Since $T^n(z)$ lives in $(\frac{1}{d}, 1) \cup \{ 0 \}$ for all $n \geqslant 2$, we have that no iterate of $z$ lies in $(\frac{k}{d}, \frac{k}{d} + \frac{1}{d^2}]$. On the other hand, $0 < z \leqslant \frac{1}{d^2}$ so the endpoints of the leaf of $M_z$ are all in the intervals $(\frac{k}{d}, \frac{k}{d} + \frac{1}{d^2}]$ for some $k$. Thus no iterate of $z$ can land on the leaf of $M_z$ and so $\Sigma_{M_z}(z)$ is well defined.

Recall that, with respect to $M_0$, the $j^{th}$ partition is of the form $I_j = (\frac{j}{d}, \frac{j+1}{d})$ while with respect to $M_z$ it is $J_j = (\frac{j}{d} + z, \frac{j+1}{d} + z)$. Since $z \leqslant \frac{1}{d^2}$, we have $T(z) \in [0, \frac{1}{d}]$. Clearly $T(z) > z$ and so $T(z) \in (z, \frac{1}{d} + z)$ making the first element of $\Sigma_{M_z}(z)$ a $0$. By Lemma \ref{cantorkneading}, to show $M_z$ has maximal entropy it suffices to check that all further iterates of $z$ never lie in the first partition $J_0 = (z, \frac{1}{d} + z)$.

By definition of $C$, we know the future iterates all live in $(\frac{1}{d}, 1]$ so the only possibility to rule out is that for some $n  \geqslant 2$, $T^n(z) \in (\frac{1}{d}, \frac{1}{d}+z)$ but this cannot happen as it would imply that $T^{n+1}(z) \in (0, dz] \subseteq (0, \frac{1}{d}]$ which we have already shown this does not happen. 

Thus $\Sigma_{M_z}(z) \in \{ 0 \} \times \{ 1, 2, ..., d-1\}^{\mathbb{N}}$ proving the claim.
\end{proof}

\begin{theorem}
\label{B2}
If $z \in [0, \frac{1}{d(d-1)}]$ and $h(M_z) = \log(2)$ then $z \in C'$.
\end{theorem}

\begin{proof}
Since $h(M_z) = \log(2)$, Lemma \ref{cantorkneading} gives that $\Sigma_{M_z}(z) \in \{ 0 \} \times \{ 1, 2, ..., d-1 \}^{\mathbb{N}}$. Thus we know that $T(z) \in (z, z+\frac{1}{d})$ and that for all $n \geqslant 2$, 
\[
T^n(z) \in [z + \frac{1}{d}, 1) \cup [0, z] \subseteq [\frac{1}{d}, 1) \cup [0, z].
\]
We claim that $T^n(z) \notin (0, z]$ for all $n \geqslant 2$. Suppose that for some $n$ this were true. Then there would be a $r$ such that:
\[
T^{n+r}(z) > z, \quad T^{n+r-1}(z) \leqslant z.
\]

But since $z \leqslant \frac{1}{d(d-1)}$ we have $dz \leqslant \frac{1}{d-1} < z + \frac{1}{d}$ so we would have $T^{n+r}(z) \in (z, z + \frac{1}{d})$ contradicting the kneading sequence of $z$.

Thus we have $T^n(z) \in [\frac{1}{d}, 1) \cup \{ 0 \}$ for all $n \geqslant 2$. By definition, this gives that $T^2(z) \in C$ and hence that $\frac{T^2(z)}{d^2} \in C'$.

Lastly, we show that $\frac{T^2(z)}{d^2} = z$. Since both are preimages of $T^2(z)$ by $T^{-2}$, and the left is the unique preimage in $(0, \frac{1}{d^2}]$, it suffices to check that $z \in (0, \frac{1}{d^2}]$.

If this were not true then we would have $\frac{1}{d^2} < z \leqslant \frac{1}{d(d-1)}$. This implies $\frac{1}{d} < T(z) < \frac{1}{d-1}$. Thus we would have that $1 < d^2z < \frac{d}{d-1}$ and hence $0 < T^2(z) < \frac{d}{d-1} - 1 = \frac{1}{d-1}$. Thus we either have $T^2(z) \in (0, z]$ or that $T^2(z) \in (z, \frac{1}{d-1}) \subseteq (z, z + \frac{1}{d})$. The first case was disproven above while the second contradicts the kneading sequence of $z$.
\end{proof}

Combining Theorems \ref{B1} and \ref{B2} we obtain Theorem B.

\section{Other Strata}

Consider now an arbitrary stratum $\Pi(s_1, s_2, ..., s_r)$ of $PM(d)$. Since our global maxima lie in $\Pi(1, 1, ..., 1)$ and we already classified $\Pi(d)$ we omit these strata. We begin with a conjectural description of the maxima.

\begin{conjecture}
Given a stratum $\Pi(s_1, ..., s_r)$ of $PM(d)$ other then $\Pi(d)$ and $\Pi(1, 1, ..., 1)$, the set of primitive majors $M$ with $h(M) = \log(r+1)$ is:
\begin{enumerate}
\item A Cantor set if $s_j > 1$ for all $j$.
\item A Cantor set union a finite number of isolated points otherwise.
\end{enumerate}
\end{conjecture}

To motivate this conjecture we consider the space of possible kneading sequences for the leaves of primitive majors with $h(m) = \log(r+1)$ in $\Pi(s_1, ..., s_r)$. Notice that, for $M = \{ L_1, ..., L_r \} \in \Pi(s_1, ..., s_r)$ to have $h(M) = \log(r+1)$ it must be the case that it is possible for two points on the unit circle to separate each leaf $L_1, L_2, ..., L_r$. After possibly relabeling the leaves, the separation vector of any such pair will either be $(L_1, ..., L_r)$ or $(L_r, ..., L_1)$. In this way we can denote $L_1$ and $L_r$ as the outer leaves of $M$. 

Applying the same approach as in sections 4, on a case by case basis we can describe the kneading sequences of each leaf $L_j$. Interestingly the analysis depends only on the outer leaves $L_1$ and $L_r$. If $L_1$ and $L_r$ are both ideal lines then, as in section 4, there will be either one or two possible kneading sequences for each leaf (depending on if $r$ is even or odd). When both of the end leaves are not an ideal line, there will be a (symbolic) Cantor set of possible kneading sequences for each leaf. In the special case where one end leaf is a line and the other is not then there will always be a Cantor set of maximal kneading sequences but when $r$ is assumed to be odd there will also be isolated sequences.

Indeed, consider such a primitive major $M = \{ L_1, L_2, ..., L_r \}$ with the leaves ordered so that $L_1$ and $L_r$ are the outer leaves of degree $s_1$ and $s_r$ respectively. Then there are intervals $I^1_1, I^1_2, ..., I^1_{s_1}$ and $I^2_1, I^2_2, ..., I^2_{s_r}$ of the unit circle minus the endpoints of each leaf $L_j$ such that (1) each $I^i_k$ is bounded by $L_1$ and $L_r$ respectively, and (2) any pair of points in the product $I^1_k \times I^2_j$ has separation vector $(L_1, L_2, ..., L_r)$. Then we can define new intervals $I_A = \cup_k I^1_k$ and $I_B = \cup_k I^2_k$. Then relative to these two new intervals, the proofs of Propositions \ref{evenprop} and \ref{oddprop} both apply (depending on if $r$ is even or odd). This give the modified kneading sequence of each leaf $L_j$ (using symbols $A$ and $B$) from which we can recover the kneading sequence by replacing being in $I_A$ by being in any choice of $I^1_k$ and being in $I_B$ with being in any $I^2_j$.

Unlike in the global case or the unicritical case we cannot hope that these kneading sequences uniquely correspond to primitive majors. Indeed, since we assume nothing about $L_2, L_3, ..., L_{r-1}$ they could be chosen in any order. Moreover, in the stratum $\Pi(s_1, s_2, ..., s_r)$ we can choose the degree of the end leaves however we want from $s_1, ..., s_r$ thus giving several choices. In fact, it is not even clear that each kneading sequence would correspond to a primitive majors kneading sequence.

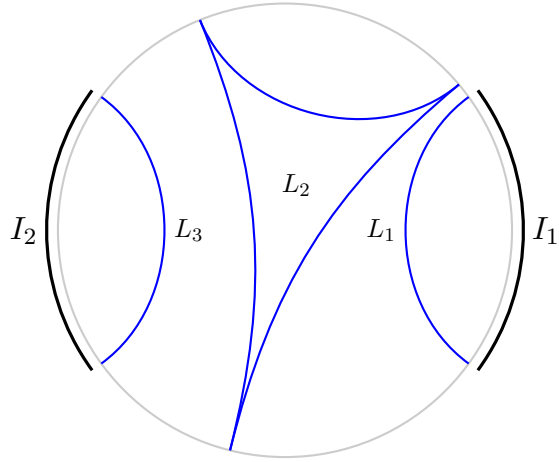
\begin{figure}
\begin{tikzpicture}[scale=3]

  % 1. Boundary Circle
  \draw[thick, gray!40] (0,0) circle (1);

  % 2. Define Coordinates for Leaf 1 (1/10 and 9/10)
  \path (9/10*360:1) coordinate (L1_A)
        (1/10*360:1) coordinate (L1_B);

  % 3. Define Coordinates for Leaf 3 (4/10 to 6/10)
  \path (4/10*360:1) coordinate (L3_A)
        (6/10*360:1) coordinate (L3_B);

  % 4. Define Coordinates for Leaf 2 - Ideal Triangle based on denominator 9
  \path (1/9*360:1)       coordinate (T_A)
        ({(1/9+1/5)*360}:1) coordinate (T_B)
        ({(1/9+3/5)*360}:1) coordinate (T_C);

  % 5. Draw Leaf 1 (diff = 72° -> bend = 54) -> Mapped to arc inside the circle
  \draw[thick, blue] (L1_A) to [bend left=54] node[pos=0.5, auto, black, scale=0.9] {$L_1$} (L1_B);

  % 6. Draw Leaf 3 (diff = 72° -> bend = 54) -> L_3
  \draw[thick, blue] (L3_A) to [bend left=54] node[pos=0.5, auto, black, scale=0.9] {$L_3$} (L3_B);

  % 7. Draw Leaf 2 - Ideal Triangle -> L_2
  \draw[thick, blue] (T_A) to [bend left=54] (T_B);
  \draw[thick, blue] (T_B) to [bend left=18] (T_C);
  \draw[thick, blue] (T_C) to [bend left=18] (T_A);
  
  % Centered label for the entire Triangle Leaf L2
  \node[black, scale=0.9] at (barycentric cs:T_A=1,T_B=1,T_C=1) {$L_2$};

  % 8. Draw Boundary Regions/Intervals on the Unit Circle
  % Region I_1: minor arc between 9/10 (324°) and 1/10 (36°)
  \draw[very thick, black] (324:1.05) arc (324:396:1.05);
  \node[black, scale=1.1] at (0:1.15) {$I_1$};

  % Region I_2: minor arc between 4/10 (144°) and 6/10 (216°)
  \draw[very thick, black] (144:1.05) arc (144:216:1.05);
  \node[black, scale=1.1] at (180:1.15) {$I_2$};

\end{tikzpicture}
\caption{A primitive major in $\Pi(1, 1, 2) \subseteq PM(5)$. For majors of this form (that is, lying in the same connected component of $\Pi(1, 1, 2)$), there are finitely many kneading sequences the major could have if ithas entropy $\log(4)$, the maximum possible for this stratum.}
\label{lam3}
\end{figure}

\begin{figure}
\begin{tikzpicture}[scale=3]

  % 1. Boundary Circle
  \draw[thick, gray!40] (0,0) circle (1);

  % 2. Define Coordinates for Leaf 1 (Ideal Triangle: 1/10, 3/10, 9/10)
  \path (36:1)   coordinate (T1_A)   % 1/10
        (108:1)  coordinate (T1_B)   % 3/10
        (324:1)  coordinate (T1_C);  % 9/10

  % 3. Define Coordinates for Leaf 2 (7/20 to 15/20)
  \path (126:1)  coordinate (L2_A)
        (270:1)  coordinate (L2_B);

  % 4. Define Coordinates for Leaf 3 (4/10 to 6/10)
  \path (144:1)  coordinate (L3_A)
        (216:1)  coordinate (L3_B);

  % 5. Draw Leaf 1 - Perfect Ideal Triangle -> L_1
  \draw[thick, blue] (T1_A) to [bend left=54] (T1_B);
  \draw[thick, blue] (T1_B) to [bend right=18] (T1_C);
  \draw[thick, blue] (T1_C) to [bend left=54] (T1_A);
  
  % Centered label for the entire Triangle Leaf L1
  \node[black, scale=0.9] at (barycentric cs:T1_A=1,T1_B=1,T1_C=1) {$L_1$};

  % 6. Draw Leaf 2 (diff = 144° -> bend left = 18) -> L_2
  \draw[thick, blue] (L2_A) to [bend left=18] node[pos=0.5, auto, black, scale=0.9] {$L_2$} (L2_B);

  % 7. Draw Leaf 3 (diff = 72° -> bend left = 54) -> L_3
  \draw[thick, blue] (L3_A) to [bend left=54] node[pos=0.5, auto, black, scale=0.9] {$L_3$} (L3_B);

  % 8. Outer Interval Labels - Inner Layer (Radius 1.10)
  % Interval I_1 = I_A: (4/10 to 6/10) -> 144° to 216°
  \draw[very thick, black] (144:1.10) arc (144:216:1.10);
  \node[black, scale=1.0, fill=white, inner sep=2pt] at (180:1.10) {$I_1 = I_A$};

  % Interval I_2: (1/10 to 3/10) -> 36° to 108°
  \draw[very thick, black] (36:1.10) arc (36:108:1.10);
  \node[black, scale=1.0, fill=white, inner sep=2pt] at (72:1.10) {$I_2$};

  % Interval I_3: (9/10 to 1/10) -> 324° to 396°
  \draw[very thick, black] (324:1.10) arc (324:396:1.10);
  \node[black, scale=1.0, fill=white, inner sep=2pt] at (0:1.10) {$I_3$};

  % 9. Outer Interval Labels - Outer Layer (Radius 1.30)
  % Interval I_B: (9/10 to 3/10) -> 324° to 468°
  \draw[very thick, black] (324:1.30) arc (324:468:1.30);
  \node[black, scale=1.0, fill=white, inner sep=2pt] at (36:1.30) {$I_B$};

  % 10. Hollow Separation Dot between I_2 and I_3 at exactly 1/10 (36°)
  \draw[thick, fill=white] (36:1.10) circle (0.018);

\end{tikzpicture}
\caption{A primitive major in $\Pi(1, 1, 2) \subseteq PM(5)$.For majors of this form (that is, lying in the same connected component of $\Pi(1, 1, 2)$), the possible kneading sequences the major could have is a Cantor set union finitely many sequences if it has entropy $\log(4)$, the maximum possible for this stratum.}
\label{lam4}
\end{figure}
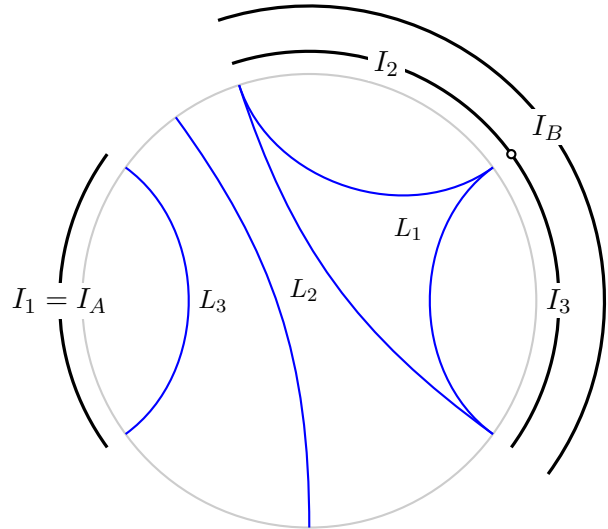

As an example, consider figures \ref{lam3} and \ref{lam4}. Both of these belong to $\Pi(1, 1, 2)$ and yet the possible kneading sequences for laminations of their forms with entropy $\log(4)$ is quite different. For figure \ref{lam3} there are two intervals which separate all three vectors, $I_1$ and $I_2$. Any pair in $I_1 \times I_2$ has separation vector $(L_1, L_2, L_3)$ hence $T(L_1)$ and $T(L_2)$ are separated, and $T(L_2)$ and $T(L_3)$ are separated. Thus $T(L_1)$ and $T(L_3)$ lie in the same interval, say $I_1$. Any totally separated pair $(x, y) \in I_1 \times I_2$ would have edges to $(T(x), T(L_1))$ and $(T(y), T(L_3))$. Since $T(L_1)$ and $T(L_3)$ lie in $I_1$ this means $T(x), T(y)$ are in $I_2$ for any such pair $(x, y)$. Thus if $x$ belongs to a totally separated vertex $(x, y)$ then its kneading sequence is either $(\overline{2})$ or $(1, \overline{2})$. $T(L_2)$ is the former case while both $T(L_1)$ and $T(L_3)$ are the latter. Since the choice of saying $T(L_1) \in I_1$ was arbitrary it is also possible the kneading sequences could be $(\overline{1})$ and $(2, \overline{1})$.

Now let us consider figure \ref{lam4}. In this case pairs in $I_1 \times I_2$ and $I_1 \times I_3$ have separation vector $(L_3, L_2, L_1)$. If we define $I_A$ as $I_1$ and $I_B$ as $I_2 \cup I_3$ then we can use the symbols $A, B$ for kneading sequences and have that $(x, y)$ has separation vector $(L_3, L_2, L_1)$ when $(x, y) \in I_A \times I_B$. Repeating the same reasoning as before we find that any totally separated pair $(x, y)$ will have kneading sequences $(\overline{A})$ and $(B \overline{A})$ or else $(\overline{B})$ and $(A, \overline{B})$, so in particular if this lamination were to maximize entropy then these are the possible kneading sequences of the leaves. But now when we replace the symbol $A$ with $1$ and $B$ with any choice of $2$ and $3$ we find that in the first case there are two types of maximal kneading sequences, either using $(\overline{1})$ and $(2, \overline{1})$, or to use $(\overline{1})$ and $(3, \overline{1})$. The second case however turns into a symbolic Cantor set with possible kneading sequences of the form $\{2, 3\}^{\mathbb{N}}$ and $\{ 1 \} \times \{2, 3 \}^{\mathbb{N}}$.


\begin{thebibliography}{99}

\bibitem[AF]{AF}  L. Alsed\`a, N. Fagella, \textit{Dynamics on Hubbard Trees}, Fund. Math. \textbf{164} (2000), no. 2, 115-141.

\bibitem[DS]{DS} D. Dudko, D. Schleicher, \textit{Core Entropy of Quadratic Polynomials}. Arnold Math J. \textbf{6}, 333–385 (2020).

\bibitem[Ga]{Ga} Y. Gao, \textit{On Thurston's Core Entropy Algorithm}, Trans. Amer. Math. Soc. \textbf{373} (2020), 747-776.

\bibitem[GT]{GT} Y. Gao, G. Tiozzo, \textit{The Core Entropy for Polynomials of Higher Degree}, J. Eur. Math. Soc. \textbf{24} (2022), no. 7, pp. 2555–2603.

\bibitem[Ju]{Ju} W. Jung, \textit{Core entropy and biaccessibility of quadratic polynomials}, available at arXiv:1401.4792[math.DS].

\bibitem[Li]{Li} T. Li, \textit{A monotonicity conjecture for the entropy of Hubbard trees}, PhD thesis, SUNY Stony Brook, 2007.

\bibitem[Po]{Po} A. Poirier, \textit{Critical Portraits for Postcritically Finite Polynomials}, Fund. Math. \textbf{203} (2009), no. 2, 107-163.

\bibitem[Ti1]{Ti1} G. Tiozzo, \textit{Topological entropy of quadratic polynomials and dimension of sections of the Mandelbrot set}, Adv. Math. \textbf{273} (2015), 651–715.

\bibitem[Ti2]{Ti2} G. Tiozzo, \textit{Continuity of Core Entropy of Quadratic Polynomials}, Invent. Math. \textbf{203} (2016), no. 3, 891–921.

\bibitem[Th+]{Th+} W. Thurston, H. Baik, Y. Gao, J. Hubbard, K. Lindsey, L. Tan, D. \textit{Thurston, Degree-d invariant
laminations}, preprint.

\end{thebibliography}
\end{document}